\documentclass[11pt]{amsart}
\usepackage{appendix}
\usepackage[colorlinks]{hyperref}
\usepackage{subfigure}
\usepackage{mathtools}
\usepackage{srcltx}

\usepackage{amsthm}
\usepackage{amsmath}%
\usepackage{amsfonts}%
\usepackage{amssymb}%
\usepackage{esint}
\usepackage{graphicx}
\usepackage{subfigure}
\usepackage{pstricks}
\usepackage{subfigure}
\usepackage{pst-plot}

\newcommand{\R}        {\mathbb {R}}

\newcommand{\eps}      {\epsilon}
\newcommand{\lap}      {\bigtriangleup}

\newcommand{\grad}     {\nabla}
\newcommand{\noi}      {\noindent}
\newcommand{\del}      {\partial}

\newtheorem{theorem}{Theorem}[section]
\newtheorem{proposition}{Proposition}[section]

\newtheorem{lemma}{Lemma}[section]
\newtheorem{corollary}{Corollary}[section]

\theoremstyle{definition} 

\newtheorem*{maintheorem*}{Main Theorem}

\allowdisplaybreaks

\numberwithin{equation}{section}
\numberwithin{figure}{section}
\numberwithin{table}{section}

\newcounter{asnr}

{\ifnum\value{asnr}=0 \stepcounter{asnr} 
  \begin{enumerate}[label=\textbf{A}.\arabic{enumi}]
    \else
    \begin{enumerate}[label=\textbf{A}.\arabic{enumi},resume] \fi}
{\end{enumerate}}

\title[]{On some Variational Problems set on domains tending to infinity}

\author[M. Chipot]{M.\ Chipot} \address[Michel Chipot]{\newline Institute f\"ur Mathematik,\ Universit\"at Z\"urich, \newline
Winterthurerstr. 190,\ CH-8057 Z\"urich,
Switzerland.}
\email[]{m.m.chipot@math.uzh.ch}

\author[A. Mojsic]{A. Mojsic} \address[Aleksandar Mojsic]{\newline Helmut Schmidt University / University of the Federal Armed Forces Hamburg, \newline
Department of Mechanical Engineering,\newline
Holstenhofweg 85\newline
22043 Hamburg, Germany.}
\email[]{mojsica@hsu-hh.de}

\author[P. Roy]{P. Roy} \address[Prosenjit Roy]{\newline Tata Institute of Fundamental Research \newline
Sharadanagar,\ GKVK Campus,
Postbox - 560065,\newline
Bangalore, India .}
\email[]{prosenjit@math.tifrbng.res.in}

\date{\today}

\begin{document}
\maketitle

ABSTRACT: \ Let $\Omega_\ell = \ell\omega_1 \times \omega_2$ where   $\omega_1 \subset \R^p$ and $\omega_2 \subset \R^{n-p}$ are assumed to be open and 
bounded. We consider the following minimization problem:
$$E_{\Omega_\ell}(u_\ell) = \min_{u\in W_0^{1,q}(\Omega_\ell)}E_{\Omega_\ell}(u)$$ 
where $E_{\Omega_\ell}(u) = \int_{\Omega_\ell}F(\grad u)-fu$, $F$ is a convex  function and $f\in L^{q'}(\omega_2)$. We are 
interested in  studying the asymptotic behavior of the solution $u_\ell$ as  $\ell$  tends to infinity.

\section{Introduction}
For $ 1 \leq p \leq n-1$ an integer, let $\Omega_\ell = \ell\omega_1 \times \omega_2 \subset \R^n$  where  $\omega_1 \subset \R^p$ and $ \omega_2 \subset \R^{n-p}$ are open and bounded. $\omega_1$ is also assumed to be star shaped with respect to the origin.  We will refer to $\Omega_\ell$ as a cylindrical domain. Points in
$\Omega_\ell$ will be denoted by $X = (X_1,X_2)$ where  $X_1 = (x_{1} , . . . , x_p) \in \ell\omega_1$ and $X_2 = (x_{p+1} , . . . , x_n) \in \omega_2$. Let $W^{-1,q'}(\omega)$
denote the dual space of the usual (cf \cite{i}) Sobolev space $W_0^{1,q}(\omega)$.
 For $q >1 $, let  $$f\in L^{q'}(\omega_2),$$ 
where 
$\frac{1}{q} +\frac{1}{q'} =1.$
\smallskip

\textit{Definition (Uniform convexity of power $q$-type)}: \  We say that a function $G :\R^n \rightarrow \R$ is a ``uniformly convex function of power $q$-type" 
if there exists a constant $\alpha= \alpha(q)$ such that 
$\forall  \xi, \eta \in \R^n$ 
\begin{equation}
\label{alpha}
2G(\frac{\xi + \eta}{2}) + \alpha |\xi -\eta|^q \leq G(\xi) + G(\eta).
\end{equation}
There is a large amount of  literature available on the study of such class of functions. We refer to \cite{bo} and the references there.  For $q \geq 2$, the function $G(x) = |x|^q$ belongs to such a class.
If $1 < q < 2$ then there does not exist any function satisfying \eqref{alpha}. This is because
for finite convex functions  it is known \cite{ff} that  its second order derivative  exists almost everywhere, which is  contradictory in this case.

Let $F :\R^n \rightarrow \R$ be a  ``uniformly convex function of power $q$-type", satisfying the following  growth condition: 
\begin{equation}
\label{quadratic}
\lambda |\xi |^q \leq F(\xi) \leq \Lambda |\xi|^q, \hspace*{3mm} \forall \xi \in \R^n,
\end{equation}
for some $\lambda, \Lambda > 0$.\smallskip

We consider the following minimization problem:
\begin{equation}\label{equaation}
E_{\Omega_\ell}(u_\ell) = \min_{u \in W_0^{1,q}(\Omega_\ell)}E_{\Omega_\ell}(u)
\end{equation}
where  
\begin{equation}\label{func}
E_{\Omega_\ell}(u) := \int_{\Omega_\ell}F(\grad u)-fu.
\end{equation}  
We refer to \cite{te} for the proof of existence and uniqueness of such $u_\ell$. In this article we are mainly interested  in studying the asymptotic 
behavior of $u_\ell$ as  $\ell$  tends to infinity. We consider the following minimization problem defined on the cross
section $\omega_2$ of $\Omega_\ell:$ 
\begin{equation}\label{equaation1}
E_{\omega_2}(u_\infty) = \min_{u \in W_0^{1,q}(\omega_2)}E_{\omega_2}(u)
\end{equation}
where  $\forall  u \in W_0^{1,q}(\omega_2),$
$$E_{\omega_2}(u) := \int_{\omega_2}F(0,\grad_{X_2} u)-fu.$$  \smallskip

In \cite{karen}, the authors considered the same problem for the particular case of  $$F(\xi) = A\xi\cdot\xi,$$ where 
 $$A :=\begin{pmatrix}A_{11}  & A_{12}\\
A_{12}^t & A_{22}
  \end{pmatrix}
$$ is $n \times n$ positive definite matrix and $``\cdot"$ denotes usual Euclidean scalar product. $A_{11}, A_{12}$ and
$A_{22}$ are respectively $p \times p, p\times (n-p)$ and $(n-p)\times (n-p)$ matrices  with bounded coefficients. $A_{12}^t$ denotes the transpose of the matrix $A_{12}$. 
It is easy to see that, in this case  $F$ satisfies \eqref{quadratic} with $q=2$. This paper can be considered as the principal incentive for our current work. 
In the case above the unique minimizer $u_\ell$ additionally satisfies the following Euler-Lagrange equation:
\begin{equation}
\label{eq:4}
\left\{
\begin{aligned}
  & - \textrm{div} \left( A\grad u_\ell\right) = f \quad \text{ in }\Omega_\ell,\\
    &u_\ell = 0 \quad \text{ on }\del\Omega_\ell,
     \end{aligned}
\right.
\end{equation}
where  ${\rm div}$ and $\grad $ denotes the divergence operator and the gradient in the $X$ variable. We recall their main result.
 Let ${\rm div_{X_2}}$ and $\grad_{X_2}$ denotes the divergence operator and the gradient in the $X_2$ variable. For $x>0$, throughout this paper $[x]$ will denote the greatest integer less than or equal to $x$.

\begin{theorem}\label{jjh}
\emph{[Chipot-Yeressian]}
There exists some constants $A, B > 0$, such that  
\begin{equation}
\label{chk}
\int_{\Omega_{\frac{\ell}{2}}} |\grad (u_\ell - u_\infty)|^2 \leq A e^{-B\ell}
\end{equation}
where $u_\infty$ is the solution to the problem
\begin{equation*}
\label{eq:5}
\left\{
\begin{aligned}
  & -{\rm div_{X_2}} \left( A_{22}\grad_{X_2} u_\infty \right) = f \quad \text{ in } \omega_2,\\
    &u_\infty = 0 \quad \text{ on }\del\omega_2.
     \end{aligned}
\right.
\end{equation*}
\end{theorem}

Their proof relies on a suitable choice of test function in the weak formulation of \eqref{eq:4} and an iteration technique. 
\smallskip 

In \cite{z}, the authors studied similar issues for the case of $F(\xi) = |\xi|^q$ where $q \geq 2$.  In their case the minimizer satisfies 
the following Euler-Lagrange equation:
\begin{equation*}
\label{eq:77}
\left\{
\begin{aligned}
  & - {\rm div }\left(|\grad u_\ell|^{q-2}\grad u_\ell \right) = f \quad \text{ in }\Omega_\ell,\\
    &u_\ell = 0 \quad \text{ on }\del\Omega_\ell.
     \end{aligned}
\right.
\end{equation*}
In this case also they obtained the convergence of $u_\ell$ toward the solution of an associated problem set on the cross section $\omega_2$.  Their main result is the following.

\begin{theorem}\label{jjr}
\emph{[Chipot-Xie]}
For $q \geq 2$, there exists some constants $A, r > 0$, such that  
\begin{equation*}
\label{chk}
\int_{\Omega_{\frac{\ell}{2}}} |\grad (u_\ell - u_\infty)|^q \leq A \ell^{-r}
\end{equation*}
where $u_\infty$ is the solution to the problem \begin{equation*}
\label{eq:5}
\left\{
\begin{aligned}
  & -{\rm div_{X_2}} \left( |\grad_{X_2} u_\infty |^{q-2}\grad u_\infty\right) = f \quad \text{ in }\omega_2,\\
    &u = 0 \quad \text{ on }\del\omega_2.
     \end{aligned}
\right.
\end{equation*}
\end{theorem}
\smallskip

We emphasize that  in our main result (Theorem \ref{main})  we do not assume any regularity on $F$ (except that the condition \eqref{quadratic} forces $F$ to be differentiable at $0$), and hence
our problem \eqref{equaation} is purely variational in nature ($u_\ell$ does not satisfy any Euler type equation). Hence this work is a generalization of the work 
\cite{karen} and \cite{z}. 
\smallskip

The following relation between the problems \eqref{equaation} and \eqref{equaation1} (for large $\ell$) is the main result of this paper. 

\begin{theorem}
\label{main}
Under the assumption $p =1$, $q \geq 2$, \eqref{alpha} and \eqref{quadratic}, if $u_\ell$ and $u_\infty$ satisfy \eqref{equaation} and \eqref{equaation1} respectively then
$$\int_{\Omega_{\frac{\ell}{2}}}|\grad(u_\ell - u_\infty)|^q \leq \frac{A}{\ell^{\frac{1}{q-1}}},$$
for some constant $A > 0$ (independent of $\ell$). 
\end{theorem}

For the case when $q=2$, note that the rate of convergence in Theorem \ref{jjh} is exponential, 
where in Theorem 1.3  it is  $O(\ell^{-1})$. But it is possible to recover an exponential rate of convergence  under an
 additional assumption on $F$, namely
 \begin{equation}
\label{alpha2}
2F(\frac{\xi + \eta}{2}) + \beta |\xi -\eta|^2 \geq F(\xi)+F(\eta) \hspace*{3mm} \forall \xi, \eta \in \R^n  \ \textrm{ and  for some}  \  \beta \geq \alpha.
\end{equation} 

 More  precisely  in this direction our result is the following.

\begin{theorem}
\label{reg}
Under the assumption \eqref{alpha},  \eqref{alpha2} and $q=2$,
if $u_\ell$ and $u_\infty$ satisfy \eqref{equaation} and \eqref{equaation1} respectively then one has
$$\int_{\Omega_{\frac{\ell}{2}}}|\grad(u_\ell - u_\infty)|^2 \leq  Ae^{-B\ell},$$
for some constant $A,  B > 0$  (independent of $\ell$).
\end{theorem}

For $q > 2$, affine functions are the only functions which satisfies \eqref{alpha2}
and hence cannot satisfy \eqref{alpha} simultenously. This justifies the condition \eqref{alpha2}. The function $F(\xi) = |\xi|^2$, satisfies \eqref{alpha2}, with an equality sign, for $\beta = \frac{1}{2}$ and hence it also satisfies  \eqref{alpha}.   It is important here to mention that \eqref{alpha}  and \eqref{alpha2} together imply that $F \in C^1(\R^n)$ (see Lemma \ref{klm}). Hence $u_\ell$ in this case satisfies the following Euler  equation:
 \begin{equation*}
\left\{
\begin{aligned}
  & -{\rm div} \left(  \grad F (\grad u_\ell ) \right) = f \quad \text{ in }\Omega_\ell,\\
    &u_\ell = 0 \quad \text{ on }\del\Omega_\ell.
     \end{aligned}
\right.
\end{equation*}
Then,  following a technique of \cite{karen} one can obtain   a different proof of Theorem \eqref{reg}. But the proof  that we present here uses only the variational structure of the problem.
\smallskip

From the point of view of applications, many problems of mathematical physics are set on large cylindrical domains. For instance, these are
porous media flows in channels, plate theory, elasticity theory, etc. 
Many problems of the type  ``$\ell \rightarrow \infty$" were  studied in the past. Apart from  second order elliptic equations \cite{karen} already mentioned, 
this includes eigenvalue problems, parabolic problems, variational inequalities, Stokes problem, hyperbolic problems and many others. We refer to 
\cite{b,sor, c,dir,d,pr,ka,z, chow, gu,g} and the references there  for the literature available in this direction.
\smallskip

The work of this paper is organized as follows. In the next section we study two problems similar to \eqref{equaation} in dimension $1$ to develop a better
 understanding of this kind of issues in a simpler situation. In section 3 we will present the proof of Theorem \ref{main}. In  section 4, we will give the  proof of 
Theorem \ref{reg}. A partial result is obtained when the cylinder goes to infinity in more than one directions [see, Theorem \ref{p > 1 }].
 We will also present some results regarding the asymptotic behavior of  $ \frac{E_{\Omega_\ell}(u_\ell)}{2\ell}$ as $\ell$ 
tends to infinity. 

\section{A one dimensional problem}
In this section we assume that $\Omega_\ell = (-\ell,\ell)$,  $F$  is a  
function satisfying  \eqref{quadratic}. The functional, analogous to \eqref{func} in one dimensional situation, is defined as 
\begin{equation*}
\label{1func}
E_{\Omega_\ell}(u) = \int_{-\ell}^\ell F(u')-\gamma u, \ u\in W_0^{1,q}(\Omega_\ell)
\end{equation*}
where $\gamma \in \R$. To observe the asymptotic behavior of the minizers $u_\ell$, we take the test case when $F(x) = \frac{x^2}{2}$. Then one can write down explicitly the solution, after solving the associated Euler-Lagrange equation, as $u_\ell(x) = \frac{\gamma}{2}(\ell^2 -x^2), \ x \in (-\ell,\ell)$. This implies that $u_\ell \rightarrow  \pm \infty$ pointwise  depending on the sign of $\gamma$.\smallskip

Now we consider the case  of a general $F$ satisfying \eqref{quadratic} and $\gamma > 0$ (we are restricting ourselves to  a constant $\gamma$ but the proof goes through for a positive function $\gamma$ bounded away from zero uniformly). We do not assume any convexity on $F$ but the existence of a minimizer $u_\ell$ that we consider next. First note  that $$u_\ell \geq 0, \hspace*{3mm}\textrm{ in} \ \Omega_\ell.$$ This is because, if $u_\ell < 0$ on $A \subset \Omega_\ell$, where $|A| > 0$. Then defining $w_\ell = \max\{0, u_\ell\}$, one gets $E_{\Omega_\ell}(w_\ell ) < E_{\Omega_\ell}(u_\ell)$, which contradicts the definition of $u_\ell$.\smallskip

We define
\begin{equation*}
\label{un}
u_\ell(x_\ell) = \max_{x \in [-\ell,\ell]} u_\ell(x).
\end{equation*}
 Note that since $u_\ell$ is a  continuous function, such a $x_\ell \in [-\ell,\ell]$ always exists.
\begin{lemma}\label{mo}
$u_\ell$ is non decreasing  on $(-\ell, x_\ell)$ and non increasing on $(x_\ell,\ell)$.
\end{lemma} 
\begin{proof} If $u_\ell$ is not non decreasing on $(-\ell, x_\ell)$ there exist $y_0 $ and $y_1$ such that 
$$
-\ell < y_0 < y_1 < x_\ell ~~~,~~~u_\ell(y_0) > u_\ell(y_1).
$$
Let $z$ be a point in $(y_0, x_\ell)$ such that $u_\ell(z)= \min_{(y_0, x_\ell)}u_\ell$ and $I_z$ the connected component of 
$\{x~| ~u_\ell(x) < u_\ell(y_0)\}$ containing $z$. 
Define then $\phi_\ell$ by
$$
 \phi_\ell = \left\{
 \begin{array}{ll}
 u_\ell \hspace{8.7mm}&\textrm{on} \ \Omega_\ell \setminus I_z, \nonumber\\
  u_\ell(y_0)  \hspace{16.9mm} &\textrm{on}  \ I_z. \nonumber
\end{array}
\right.
$$
Clearly this function satisfies $E_{\Omega_\ell}(\phi_\ell) < E_{\Omega_\ell}(u_\ell)$ which contradicts the definition of $u_\ell$. The proof that $u_\ell$ is non increasing on $(x_\ell,\ell)$ follows the same way.
\end{proof}

\begin{theorem}
$u_\ell(x) \rightarrow \infty$ pointwise $\forall x$, and for fixed $a < b$, one has 
$\int_a^bu_\ell^s \rightarrow \infty,$ for all $s> 0$. 
\end{theorem}
\begin{proof}
Let us argue by contradiction. Suppose that there exist  $K > 0$, $ z \in \R$ and a sequence $l_k \rightarrow \infty$ (which is again labeled by $\{\ell\}$), such that  $0 \leq u_{\ell}(z) \leq K$. Let us assume that  $x_{\ell} \leq z$ for all $\ell$. This implies from the previous lemma that $u_{\ell}(x) \leq K, \ x\geq z.$  Let $\delta > 0$, define the function 
$$
 \phi_{\ell}(x) = \left\{
 \begin{array}{ll}
   u_{\ell} \hspace{8.7mm}&\textrm{on} \ I_1:=(-\ell, z], \nonumber\\
  u_\ell(z) + (x-z)(u_{\ell}(z+1)-u_{\ell}(z)+\delta)   &\textrm{on}  \ I_2:=(z, z+1], \nonumber\\
 u_{\ell}(x)+ \delta  &\textrm{on} \ I_3:= [z+1,\ell-1), \nonumber\\
  (\ell-x)(u_\ell(\ell-1)+\delta) \ &\textrm{on} \  I_4:=(\ell-1,\ell].
\end{array}
\right.
$$
It is easy to see that 
\begin{equation}
\label{dd}
E_{I_1}(u_\ell) - E_{I_1}(\phi_\ell)= 0, \  E_{I_3}(\phi_\ell) - E_{I_3}(u_\ell)\leq - \gamma \delta(\ell- z -2).
\end{equation}
Now using the fact that $0\leq\phi_\ell \leq K + \delta $ on  $I_2\cup I_4,$ it is posible to find a constant $C > 0$ (independent of $\ell$) such that
\begin{equation}
\label{oo}
 E_{I_2 \cup I_4}(\phi_\ell) - E_{I_2\cup I_4}(u_\ell) \leq C.
\end{equation}
 Adding \eqref{dd} and \eqref{oo}, we obtain, when $\ell$ is sufficiently large 
 $$ E_{\Omega_\ell}(\phi_\ell) - E_{\Omega_\ell}(u_\ell) \leq C -\gamma\delta(\ell- z -2) < 0,$$
 which is a contradiction to the definition of $u_\ell$. 
 If $x_\ell \geq  z$ for some $\ell$ large defining $\phi_\ell$ analogously on the other side of $x_\ell$ we arrive to the same contradiction and the proof follows.\smallskip

Then, using the monotonicity property of $u_\ell$, one has
$$\int_a^b u_\ell^s \geq \min \{u_\ell(a), u_\ell(b)\}^s(b-a) \rightarrow \infty.$$
This completes the proof of the theorem.
\end{proof}
\bigskip

The situation can be different if $\gamma=0$ and if some coerciveness is added (compare to  \cite{ka2}). 
Suppose $\alpha, \beta > 0$, define  the energy 
functional, $E_{\Omega_\ell} : W_{\alpha,\beta}^{1,q}(\Omega_\ell) \rightarrow \R$  as 
$$E_{\Omega_\ell}(u) = \int_{-\ell}^\ell F(u') + |u|^q$$
where $ W_{\alpha,\beta}^{1,q}(\Omega_\ell) := \left \{ u \in W^{1,q}(\Omega_\ell) \ \big| \ u(-\ell) =  \alpha, u(\ell) = \beta \right\}$. Consider then the following minimization problem:
\begin{equation}\label{jjjjj}
E_{\Omega_\ell}(v_\ell) = \inf_{u \in W_{\alpha,\beta}^{1,q}(\Omega_\ell)} E_{\Omega_\ell}(u).
\end{equation} For existence and uniqueness of such function $v_\ell$, we  refer to \cite{te} but we just assume here existence of some minimizer $v_\ell$ that we consider next. 

\begin{lemma}\label{one}
It holds that
$$ 0 \leq v_\ell \leq \max\{\alpha, \beta\}  \  \ \textrm{and} \ \int_{-\ell}^{\ell} |v_\ell'|^q + v_\ell^q \leq C $$
where $C$ is a positive constant, independent of $\ell.$
\end{lemma}
\begin{proof} Let $| \{ v_\ell < 0 \}| > 0$. Define $w_\ell = \max\{0, v_\ell\}$. Then $E_{\Omega_\ell}(w_\ell ) < E_{\Omega_\ell}(v_\ell)$, 
which contradicts the definition of $v_\ell$.  The fact that $v_\ell \leq  \max \{ \alpha, \beta\}$ follows with a similar argument,
 with  the function $W_\ell = \min\{v_\ell, \max\{\alpha, \beta\} \}$ playing the role of $w_\ell$ in the previous part.\smallskip

Clearly the function 
$$
 \phi_\ell(x_1) = \left\{
 \begin{array}{ll}
  -\alpha x_1 + \alpha -\alpha\ell \hspace{8.7mm}&\textrm{on} \ (-\ell, -\ell + 1], \nonumber\\
  0  \hspace{16.9mm} &\textrm{on}  \ (-\ell+1, \ell-1), \nonumber\\
  \beta x_1 +\beta -\ell\beta  \hspace{9.8mm}&\textrm{on} \ [\ell-1,\ell), 
\end{array}
\right.
$$
is in the space $W_{\alpha,\beta}^{1,q}(\Omega_\ell)$. Then  it holds also that for all $\ell \geq 1$, $$E_{\Omega_\ell}(\phi_\ell) = E_{\Omega_1}(\phi_1).$$ Using $\phi_\ell$ as test function in
 \eqref{jjjjj}, one gets $$E_{\Omega_\ell}(v_\ell) \leq E_{\Omega_\ell}(\phi_\ell) = E_{\Omega_1}(\phi_1).$$  This proves the lemma.
\end{proof}
Next we state the main theorem of this section which  shows exponential rate of convergence of the solutions $v_\ell$ towards $0$. The method of proof depends on an iteration scheme.
\begin{theorem}
It holds that
$$\int_{-\frac{\ell}{2}}^{\frac{\ell}{2}} |v_\ell^{'}|^q +  v_\ell^q \leq C e^{-\alpha\ell},$$
where $C$ and 
$ \alpha$ are some positive constants, independent of $\ell$. 
\end{theorem}
\begin{proof}
Suppose $\ell_1 \leq \ell$. Consider the  following test function $\phi_{\ell_1}$ defined on $\Omega_\ell = (-\ell,\ell)$ by 
$$
 \phi_{\ell_1}(x_1) = \left\{
 \begin{array}{ll}
1 \hspace{5.7mm}&\textrm{on} \ [-\ell, -\ell_1]\cup[\ell_1,\ell], \nonumber\\
-x_1  -\ell_1 + 1  \hspace{9.8mm}&\textrm{on} \ (-\ell_1,-\ell_1+1),\nonumber\\
  x_1  -\ell_1 + 1  \hspace{9.8mm}&\textrm{on} \ (\ell_1-1,-\ell_1),\nonumber\\
  0  \hspace{16.9mm} &\textrm{on}  \ [-\ell_1+1, \ell_1-1].
\end{array}
\right.
$$
The graph of $\phi_{\ell_1}$ is shown below.
 \begin{center}
\setlength{\unitlength}{1cm}
\begin{picture}(8,4)(-3.5,-1.25)
 \put(-5,0){\vector(1,0){10.4}}
 \put(0,-.4){$0$}
 \put(5.5,0){$x_1$}
 \put(0,0){\vector(0,1){2}}
 \put(3.5,0){\circle*{.1}}
 \put(2.5,0){\circle*{.1}}
 \put(-3.5,0){\circle*{.1}}
 \put(4.8,0){\circle*{.1}}
 \put(-4.8,0){\circle*{.1}}
 \put(3.4,-.4){$\ell_1$}
 \put(2.1,-.4){$\ell_1-1$}
 \put(-3,-.4){$-\ell_1+1$}
 \put(-3.9,-.4){$-\ell_1$}
 \put(4.5,-.4){$\ell$}
 \put(-5,-.4){$-\ell$}
 \put(-3.54,1.04){\line(-1,0){1.24}}
  \put(3.54,1.04){\line(1,0){1.24}}
 \put(2.5,0){\line(1,1){1.035}}
 \put(-2.5,0){\line(-1,1){1.035}}
 \put(-2.5,0){\circle*{.1}}
\put(-.5,-1.2){$\textrm{figure 1}$}
\end{picture}
\end{center}
Clearly we have $\phi_{\ell_1}v_\ell \in W_{\alpha,\beta}^{1,q}(\Omega_\ell)$. Hence from \eqref{jjjjj}, we get $E_{\Omega_\ell}(v_\ell) \leq E_{\Omega_\ell}(\phi_{\ell_1}v_\ell)$. Since $\phi_{\ell_1} = 1$ on the set $\Omega_\ell \setminus \Omega_{\ell_1}$, we have 
$$\int_{\Omega_{\ell_1}}F(v_\ell^{'}) + 	v_\ell^q \leq \int_{\Omega_{\ell_1}}F((\phi_{\ell_1}v_\ell)') + 	(\phi_{\ell_1}v_\ell)^q.$$
Setting $D_{\ell_1} := \Omega_{\ell_1}\setminus \Omega_{\ell_1-1}$ and observing that $\phi_{\ell_1}=  0 $ on $\Omega_{\ell_1-1}$, we have 
\begin{equation*}\label{ss}
\int_{\Omega_{\ell_1}}F(v_\ell^{'}) + 	v_\ell^q\leq \int_{D_{\ell_1}}F((\phi_{\ell_1}v_\ell)') + 	(\phi_{\ell_1}v_\ell)^q.
\end{equation*}
Using \eqref{quadratic}, convexity of the function $x^q$ and properties of $\phi_{\ell_1}$ we have for some constant $ C > 0$ (independent of $\ell$ )
\begin{eqnarray*} \min \{ \lambda, 1 \} \left( \int_{\Omega_{\ell_1}} |v_{\ell}'|^q + v_\ell^q \right) &\leq & \Lambda \int_{D_{\ell_1}} |(\phi_{\ell_1}v_\ell)'|^q +\int_{D_{\ell_1}} (\phi_{\ell_1}v_\ell)^q \nonumber\\
&\leq &   \Lambda \int_{D_{\ell_1}} |\phi_{\ell_1} v_{\ell}^{'} + v_\ell \phi_{\ell_1}^{'}|^q  + \int_{D_{\ell_1}} v_\ell^q \nonumber\\
&\leq & C\int_{D_{\ell_1}}  |v_{\ell}'	 |^q + v_\ell^q.
\end{eqnarray*}
The above inequality implies for  $\tilde{\Lambda} = \frac{C}{\min \{ \lambda, 1 \}}$
 \begin{equation*}
  \int_{\Omega_{\ell_1-1}}|v_{\ell}'|^q + v_\ell^q \leq \frac{\tilde{\Lambda}}{\tilde{\Lambda}+ 1}\left(\int_{\Omega_{\ell_1}}|v_{\ell}'|^q + v_\ell^q\right).
\end{equation*}
Now choosing 
$\ell_1 = \frac{\ell}{2} + 1 ,   \frac{\ell}{2}  +2, \  . \  . \  . \  ,    \frac{\ell}{2} + \textrm{[}\frac{\ell}{2}\textrm{]} $ and iterating the above formula, we obtain
\begin{equation*}
  \int_{\Omega_{\frac{\ell}{2}}}|v_{\ell}'|^q + v_\ell^q \leq \left(\frac{\tilde{\Lambda}}{1 +\tilde{\Lambda}}\right)^{\textrm{[}\frac{\ell}{2} \textrm{]}}\left( \int_{\Omega_{\frac{\ell}{2} + \textrm{[}\frac{\ell}{2}\textrm{]}}}|v_{\ell}'|^q + v_\ell^q \right).
\end{equation*}
since 
$\frac{\ell}{2}-1 <\textrm{[}\frac{\ell}{2} \textrm{]} \leq \frac{\ell}{2}$ there are positive constants $C_1$ and $C_2$  such that
\begin{equation*}
 \int_{\Omega_{\frac{\ell}{2}}} |v_{\ell}'|^q + v_\ell^q \leq 
C_1e^{-C_2 \ell}\left(\int_{\Omega_{\ell}}|v_{\ell}'|^q + v_\ell^q \right).
\end{equation*}
The theorem then follows from the last lemma.
\end{proof}

\section{Proof of Theorem \ref{main}}
We define the spaces 
$$W_{lat}^{1,q}(\Omega_\ell) := \left\{ u \in W^{1,q}(\Omega_\ell) \ \big| \  u = 0 \ \textrm{on} \ \ell\omega_1\times\del\omega_2\right\}$$
and 

\[
V^{1,q}(\Omega_{\ell})= \left\{ u\in W_{lat}^{1,q}(\Omega_{\ell}) \ \big| \  \begin{array}{c}

u(Y,X_2)
= u(Z,X_2),\\
a.e.  (Y,X_2),(Z,X_2 ) \in \del(\ell\omega_1)\times\omega_2
\end{array}\right\}.
\]
\smallskip

We consider the following minimization problem on $V^{1,q}(\Omega_\ell)$:
\begin{equation}
\label{w}
E_{\Omega_\ell}(w_\ell) = \min_{u \in V^{1,q}(\Omega_\ell)} E_{\Omega_\ell}(u),
\end{equation}
where $E_{\Omega_\ell}$ is defined as in \eqref{equaation}.
Existence and uniqueness of such $w_\ell$ again follows  as for the case of Problem \eqref{equaation}, assuming \eqref{quadratic} and $F$ strictly convex. 

\begin{theorem}\label{wl}
One has  $\forall \ell$, $w_\ell = u_\infty$, where 
$u_\infty$ is as in \eqref{equaation1}.
\end{theorem}
\begin{proof}
  Let $u \in V^{1,q}(\Omega_\ell)$.
Set
$$v(X_2) = \fint_{\ell\omega_1} u(.,X_2) := \frac{1}{|\ell\omega_1| } \int_{\ell\omega_1} u(.,X_2)$$
where $|\ell\omega_1|$ denotes the measure of $\ell\omega_1$.
It is easy to see that $v \in W_0^{1,q}(\omega_2).$ Thus 
$$E_{\omega_2}(u_\infty) \leq E_{\omega_2}(v) = \int_{\omega_2} F\left(0,\grad_{X_2} v\right) -fv.$$
Now using the fact that $u(.,X_2)$ is constant on the set $\del(\ell\omega_1)$ and from  the divergence theorem, one has 
$$0 = \fint_{\ell\omega_1} \nabla_{X_1} u.$$
Also by differentiation under the integral, we get
$$\nabla_{X_2} v = \fint_{\ell\omega_1} \nabla_{X_2} u.$$
Therefore we have by Jensen's inequality 
\begin{eqnarray*}
E_{\omega_2}(u_\infty)\leq 
\int_{\omega_2} \left\{ F \left( \fint_{\ell\omega_1}\nabla_{X_1} u, \ \fint_{\ell\omega_1}\grad_{X_2} u \right)  \right\}   - \int_{\omega_2} f \left\{ \fint_{\ell\omega_1}
u \right\}  \\
\leq  \int_{\omega_2}\fint_{\ell\omega_1} \left\{F(\nabla u)-fu \right\} = \frac{E_{\Omega_\ell}(u)}{|\ell\omega_1|}.
\end{eqnarray*}
Clearly equality holds in the above inequality if $u = u_\infty$. Then the claim follows from the uniqueness of the minimizer.
\end{proof}

 Next we consider the case when $p=1$,  which means that the cylinder $\Omega_\ell$ becomes unbounded in one direction. Points in $\omega_1$  are now
 simply denoted by $x_1$. $\grad , \grad_{X_2}$ denote the gradients in $X$ and $X_2$ variables respectively. By 
$$|u|_{q,\Omega} = \left( \int_{\Omega} |u|^q \right)^{\frac{1}{q}},$$
we will denote the $L^q$ norm of $u$ on any domain $\Omega$.\smallskip
 
 We  need a few preliminaries before proceeding to the proof of Theorem \ref{main}. We know that a convex function  is locally Lipschitz continuous on $\R^n$, 
 the next lemma provides an estimate on the growth of the Lipschitz constant for convex functions satisfying  \eqref{quadratic}.

\begin{proposition}
\textup{\label{cor:F- lipsic}For convex function $F:\mathbb{R}^{n}\rightarrow\mathbb{\mathbb{R}}$
satisfying \eqref{quadratic}, for every $P,Q\in\mathbb{R}^{n},$ we
have
\begin{multline}
\left|F\left(Q\right)-F\left(P\right)\right|\leq 2^q\Lambda\max\left\{ \left|P\right|^{q-1},\left|Q\right|^{q-1}\right\} \left|Q-P\right| \\ \leq 2^q\Lambda 
(|P|^{q-1}+|Q|^{q-1})|P-Q|.\label{eq:Lipschitz estimate}
\end{multline}
}\end{proposition}
\begin{proof}
First note that the second inequality holds trivially. For a convex function $f:\mathbb{R}\rightarrow\mathbb{R}$ and $a<b<c,$
we known that 
\begin{equation}
\frac{f\left(b\right)-f\left(a\right)}{b-a}\leq\frac{f\left(c\right)-f\left(b\right)}{c-b}.\label{eq:convex inequality}
\end{equation}
The function $f:\mathbb{R}\rightarrow\mathbb{R}$ defined as
\[
f\left(t\right)=F\left(\left(1-t\right)P+tQ\right),
\]
is clearly convex. Using the inequality \eqref{eq:convex inequality}
for $a=0$, $b=1$ and $c=t>1$ we have

\begin{equation*}
F\left(Q\right)-F\left(P\right)\leq\frac{F\left(\left(1-t\right)P+tQ\right)-F\left(Q\right)}{t-1}\leq\frac{\Lambda\left|\left(1-t\right)P+tQ\right|^{q}}{t-1}.
\end{equation*}
Now for $t=1+\frac{\left|Q\right|}{\left|Q-P\right|}$, we have 
\begin{eqnarray*}
F\left(Q\right)-F\left(P\right)&\leq &\frac{\Lambda\left|Q+\frac{\left|Q\right|}{\left|Q-P\right|}\left(Q-P\right)\right|^{q}}{\left|Q\right|}\left|Q-P\right|
\\
&\leq &\frac{\Lambda\left(\left|Q\right|+\frac{\left|Q\right|}{\left|Q-P\right|}\left|Q-P\right|\right)^{q}}{\left|Q\right|}\left|Q-P\right|\\
 &\leq & 2^q\Lambda\left|Q\right|^{q-1}\left|Q-P\right|.
\end{eqnarray*}
Now changing the roles of $P$ and $Q$, we have
\[
F\left(P\right)-F\left(Q\right)\leq 2^q\Lambda\left|P\right|^{q-1}\left|P-Q\right|,
\]
and \eqref{eq:Lipschitz estimate} follows.
\end{proof}
\smallskip

The following result is well known \cite{b}, we state it without  proof. 

\begin{lemma}\emph{[The Poincar\'e's Inequality for $\omega_2$]}
Let $0< s < t$  and $q \geq 1,$ then there exists $\lambda_1 = \lambda_1(q,\omega_2) > 0$ (independent of $s$ and $t$) such that 
$$|u|_{q,(s,t)\times \omega_2} \leq \lambda_1 ||\grad u||_{q,(s,t)\times \omega_2}, $$
for all $u \in W^{1,q}((s,t)\times \omega_2)$ with $u = 0$ on $(s,t)\times \del\omega_2$.
\end{lemma}
\smallskip

Now our main aim is to prove the Corollary \ref{za svaki skup} below, which we will use in the proof of
Theorem \ref{main}. 

For $s<t,$ $\rho_{s,t}=\rho_{s,t}(x_1)$ denotes a real, Lipschitz continuous function such that 
\begin{equation}\label{rho}
0\leq \rho_{s,t}\leq 1,~ |\rho'_{s,t}| \leq C,~\rho_{s,t}=1 \text{ on } (s,t),~ \rho_{s,t}=0 \text{ outside } (s-1,t+1)
\end{equation}
and $D_{s,t}$ denotes the set defined as 
$$
D_{s,t}=\left( \left((s-1,t+1)\times \omega_2\right) \backslash \left((s,t)\times \omega_2\right) \right) \cap \Omega_\ell.
$$

\begin{proposition}\label{new}
Let $p=1$.  There exists a constant (independent of $\ell$) such that for   $ -\ell < s < t < \ell$ we have
$$E_{(s,t)\times\omega_2}(u_\ell) \leq C |f|_{q',\omega_2}^\frac{q}{q-1}.$$ 
In particular it implies that for $\ell_0 < \ell,$
\begin{equation}\label{1}
E_{D_{\ell_{0}}}(u_{\ell}) \leq 2C|f|_{q',\omega_2}^\frac{q}{q-1}
\end{equation}
where $D_{\ell_0} :=  \Omega_\ell \cap \left( \Omega_{\ell_0+1} \setminus \Omega_{\ell_0}\right)$ and $u_\ell$ is as in \eqref{equaation}.
\end{proposition}

\begin{proof} First we claim that if $E_{D_{s,t}}(u_\ell)\leq 0$ then 

\begin{equation}\label{ds}
E_{(s,t)\times \omega_2}(u_\ell) \leq C(\lambda, \lambda_1, \Lambda)|f|_{q',\omega_2}^{\frac{q}{q-1}}.
\end{equation} 
Indeed if $E_{D_{s,t}}(u_\ell) \leq 0$  then one has
\begin{multline*}
\lambda \int_{D_{s,t}} |\grad u_\ell|^q   \leq  \int_{D_{s,t}}F(\nabla u_\ell)  \leq \int_{D_{s,t}}fu_\ell \leq |f|_{q',D_{s,t}}|u_\ell|_{q, D_{s,t}} \\
\leq \lambda_1 |f|_{q',D_{s,t}}||\nabla u_\ell||_{q, D_{s,t}}
\end{multline*}
by the Poincar\'e inequality for $\omega_2$. It follows that
\begin{equation}\label{inn}
|| \nabla u_\ell||_{q, D_{s,t}}^q \leq \left(\frac{\lambda_1}{\lambda}\right)^{\frac{q}{q-1}}|f|_{q',D_{s,t}}^{\frac{q}{q-1}}.
\end{equation}
Since $(1- \rho_{s,t})u_\ell \in W_0^{1,q}(\Omega_\ell)$ we have 
$$ E_{\Omega_\ell}(u_\ell) \leq  E_{\Omega_\ell}((1-\rho_{s,t})u_\ell), $$ 
that is to say
\begin{multline*}
E_{(s,t)\times \omega_2}(u_\ell) + E_{D_{s,t}}(u_\ell) + E_{\Omega_\ell \backslash (s-1,t+1)\times \omega_2 }(u_\ell)  \leq  E_{(s,t)\times \omega_2}((1-\rho_{s,t})u_\ell) \\+ E_{D_{s,t}}((1-\rho_{s,t})u_\ell)  +
E_{{\Omega_\ell \backslash (s-1,t+1)\times \omega_2 }} ((1-\rho_{s,t})u_\ell).~~~~~~~~~~~~~~~~~~
\end{multline*}
Since $\rho_{s,t} = 0$ outside $(s-1,t+1)$ $\rho_{s,t} = 1$ on $(s,t)$ we get 
\begin{eqnarray*}
E_{(s,t)\times \omega_2}(u_\ell) &\leq & E_{D_{s,t}}((1-\rho_{s,t})u_\ell) - E_{D_{s,t}}(u_\ell) \nonumber 
 \\
 &= & \int_{D_{s,t}}F(\nabla(1-\rho_{s,t})u_\ell) - F(\nabla u_\ell) + f\rho_{s,t} u_\ell \nonumber \\
 &\leq & \Lambda \left\{ ||\nabla (1-\rho_{s,t})u_\ell ||_{q, D_{s,t}}^q + ||\nabla u_\ell ||_{q, D_{s,t}}^q \right\} + |f|_{q',D_{s,t}}|u_\ell|_{q,D_{s,t}} \nonumber \\
 &\leq & \Lambda \left\{ ||(1-\rho_{s,t})\nabla u_\ell - u_\ell \nabla \rho_{s,t} ||_{q, D_{s,t}}^q + ||\nabla u_\ell ||_{q, D_{s,t}}^q \right\} \\&+&\lambda_1 |f|_{q', D_{s,t}}||\nabla u_\ell 
 ||_{q, D_{s,t}}.
\end{eqnarray*}
Now using the triangle inequality and Poincar\'e inequality for $\omega_2$ we obtain for some constants $K(q)$ and  $C(\lambda, \lambda_1, \Lambda, q)$,
\begin{eqnarray*}
E_{(s,t)\times \omega_2}(u_\ell) &\leq & \Lambda  K\left\{ |u_\ell|_{q, D_{s,t}}^q +  ||\nabla u_\ell||_{q, D_{s,t}}^q \right\}+\lambda_1 |f|_{q', D_{s,t}}|\nabla u_\ell|_{q, D_{s,t}} \nonumber \\
 &\leq & C(\lambda, \lambda_1, \Lambda, q) |f|_{q',\omega_2}^{\frac{q}{q-1}}
\end{eqnarray*}
by \eqref{inn}. This completes the claim i.e. shows \eqref{ds}.
\smallskip

Next we can complete the proof of the theorem. Indeed, let $m$ be the first nonnegative integer such that $D_{s-m, t+m}$ is non empty and
$$E_{D_{s-m,t+m}}(u_\ell) \leq 0. $$
If there is no such integer, 
then 
$$E_{(s,t)\times \omega_2}(u_\ell) \leq E_{(s-1,t+1)\times \omega_2}(u_\ell) \leq . . . \leq E_{\Omega_\ell}(u_\ell) \leq 0.$$
In the last inequality we used the fact that 
$$E_{\Omega_\ell}(u_\ell) \leq E_{\Omega_\ell}(0) = 0.$$
If such an $m$ exists then, by the first part of this proof,
$$E_{(s,t)\times \omega_2}(u_\ell) \leq E_{(s-m,t+m)\times \omega_2}(u_\ell) \leq C |f|_{q',\omega_2}^\frac{q}{q-1}.$$
This proves the proposition.
\end{proof}
\smallskip

\begin{corollary}\emph{[Gradient Estimate]}
\label{za svaki skup}For every $\ell_{0}\leq \ell$ we have, 
\begin{equation*}
\int_{D_{\ell_{0}}}\left|\nabla u_{\ell}\right|^{q}\leq \tilde{C},
\end{equation*}
for some constant $\tilde{C} = \tilde{C}(\lambda, \lambda_1, \Lambda, f)$  \end{corollary}
\begin{proof}
Applying H\"older's, Young's inequality  and \eqref{quadratic}, to \eqref{1}, we obtain
\begin{eqnarray*}
\lambda\int_{D_{\ell_0}}|\grad u_\ell|^q 
&\leq & \int_{D_{\ell_0}} F(\grad u_\ell) \\
&\leq & 2C|f|_{q',\omega_2}^{q'} + \int_{D_{\ell_0}}fu_\ell \\
&\leq & 2C|f|_{q',\omega_2}^{q'} +\frac{1}{\eps^{q'}q'} |f|_{q',D_{\ell_0}}^{q'} +\frac{\eps^q}{q} |u_\ell|_{q,D_{\ell_0}}^q \\
&\leq & 2C|f|_{q',\omega_2}^{q'} +\frac{1}{\eps^{q'}q'} |f|_{q',D_{\ell_0}}^{q'} +\frac{\lambda_1^q\eps^q}{q} ||\grad u_\ell||_{q,D_{\ell_0}}^q. 
\end{eqnarray*}
This implies that 
$$(\lambda - \frac{\lambda_1^q\eps^q}{q} )\int_{D_{\ell_0}}|\grad u_\ell|^q  \leq 
2C|f|_{q',\omega_2}^{q'} +\frac{1}{\eps^{q'}q'} |f|_{q',D_{\ell_0}}^{q'} .$$
The corollary follows after choosing $\eps $  such that $\lambda - \frac{\lambda_1^q\eps^q}{q} =\frac{1}{2}$.
\end{proof}
\bigskip


\textbf{Proof of Theorem \ref{main}} \hspace*{3mm} For some $0\leq \ell_0 \leq \ell-1,$ we define the  functions
\begin{equation}
\begin{array}{c}
\hat{\psi}_{\ell_0, \ell}= u_{\ell}+\frac{\rho_{\ell_0}}{2}\left(u_{\infty}-u_{\ell}\right) \  \textrm{and}
\\
\check{\psi}_{\ell_0,\ell} = u_{\infty}+\frac{\rho_{\ell_0}}{2}\left(u_{\ell}-u_{\infty}\right),
\end{array}\label{eq:Prosenjitlinear combination}
\end{equation}
where $\rho_{\ell_0}$ is defined before in \eqref{rho} with $s =-\ell_0$ and $t = \ell_0$. Because $\hat{\psi}_{\ell_0,\ell}\in W_{0}^{1,q}\left(\Omega_{\ell}\right),$
we have from \eqref{equaation}
\begin{equation}
\begin{array}{c}
E_{\Omega_{\ell}}\left(u_{\ell}\right)\leq E_{\Omega_{\ell}}\left(\hat{\psi}_{\ell_0,\ell}\right)\end{array}.\label{eq:Prosenjit1}
\end{equation}
Also since $\check{\psi}_{\ell_0,\ell}\in V^{1,q}(\Omega_{\ell}),$ from Theorem \ref{wl}
\begin{equation}
E_{\Omega_{\ell}}\left(u_{\infty}\right)\leq E_{\Omega_{\ell}}\left(\check{\psi}_{\ell_0,\ell}\right).\label{eq:Prosenjit2}
\end{equation}
From \eqref{eq:Prosenjitlinear combination} it is easy to see that
\begin{equation}\label{same}
\hat{\psi}_{\ell_0,\ell} + \check{\psi}_{\ell_0,\ell} = u_{\ell} + u_{\infty}.
\end{equation}

Adding \eqref{eq:Prosenjit1} and \eqref{eq:Prosenjit2}, we have
\begin{equation*}
E_{\Omega_{\ell}}\left(u_{\ell}\right)+ E_{\Omega_{\ell}}\left(u_{\infty}\right)\leq E_{\Omega_{\ell}}(\grad \hat{\psi}_{\ell_0,\ell})+E_{\Omega_{\ell}}\left(\grad\check{\psi}_{\ell_0,\ell}\right)
\end{equation*}
which implies 
$$\int_{\Omega_\ell} F(\grad u_\ell) + F(\grad u_\infty) \leq \int_{\Omega_\ell} F(\hat{\psi}_{\ell_0,\ell}) + F(\check{\psi}_{\ell_0,\ell}).$$
Using the fact that 
$
u_{\ell} = \hat{\psi}_{\ell_0,\ell} \mbox{ and } u_{\infty} = \check{\psi}_{\ell_0,\ell}$ on $\Omega_{\ell}\setminus\Omega_{\ell_0 +1},$ 
we get 
\begin{equation*}\label{hj}
\int_{\Omega_{\ell_0 +1}} F\left(\nabla u_{\ell}\right)+F\left(\nabla u_{\infty}\right)\leq\int_{\Omega_{\ell_0 +1}}F(\nabla\hat{\psi}_{\ell_0,\ell})+F\left(\nabla\check{\psi}_{\ell_0,\ell}\right).
\end{equation*}
Note that  $\hat{\psi}_{\ell_0,\ell}=\check{\psi}_{\ell_0,\ell}=\frac{u_{\ell} + u_{\infty}}{2}$ on $\Omega_{\ell_0},$ which gives 
\begin{multline}\label{pp}
\int_{\Omega_{\ell_0}}F\left(\nabla u_{\ell}\right)+F\left(\nabla u_{\infty}\right)-2F\left(\frac{\nabla u_{\ell}+\nabla u_{\infty}}{2}\right)
\\
\leq\int_{D_{\ell_0}}F(\nabla\hat{\psi}_{\ell_0,\ell})+ F\left(\nabla\check{\psi}_{\ell_0,\ell}\right)-F\left(\nabla u_{\ell}\right)-F\left(\nabla u_{\infty}\right) := I.
\end{multline}
Using the convexity of $F$ and \eqref{same}, we have 
\begin{multline*}
I = \int_{D_{\ell_0}}F(\nabla\hat{\psi}_{\ell_0,\ell}) + F\left(\nabla\check{\psi}_{\ell_0,\ell}\right)- F\left(\nabla u_{\ell}\right)- F\left(\nabla u_{\infty}\right)\\
\leq \int_{D_{\ell_0}} F(\nabla\hat{\psi}_{\ell_0,\ell}) + F\left(\nabla\check{\psi}_{\ell_0,\ell}\right)- 2F\left(\frac{\nabla u_{\ell}+\nabla u_{\infty}}{2}\right)
\\
=\int_{D_{\ell_0}} F(\nabla\hat{\psi}_{\ell_0,\ell})+F\left(\nabla\check{\psi}_{\ell_0,\ell}\right)-
2F\left(\frac{\nabla\hat{\psi}_{\ell_0,\ell}+\nabla\check{\psi}_{\ell_0,\ell}}{2}\right).
\end{multline*}
Now using Proposition \ref{cor:F- lipsic}, we obtain for some constant $C=C(q,\Lambda) > 0,$
\begin{multline}\label{mu}
|I| \leq  C\int_{D_{\ell_0}}\left\{\big|\nabla\hat{\psi}_{\ell_0,\ell}\big|^{q-1}+ \big|\nabla(\hat{\psi}_{\ell_0,\ell}+ \check{\psi}_{\ell_0,\ell})\big|^{q-1}+
\left|\nabla\check{\psi}_{\ell_0,\ell}\right|^{q-1}\right\}\big|\nabla(\hat{\psi}_{\ell_0,\ell}-\check{\psi}_{\ell_0,\ell})\big|.
\end{multline}
Since $q \geq 2$, using the monotonicity of the function $|X|^{q-1}$, we have for $a, ~b>0$
$$
(a+b)^{q-1} \leq (2 \max\{a,b\})^{q-1}=2^{q-1}\max \{a,b\}^{q-1}\leq 2^{q-1}(a^{q-1}+b^{q-1}).
$$
Now from \eqref{mu} we get,
$$I \leq C\int_{D_{\ell_0}} \left(\big|\nabla\hat{\psi}_{\ell_0,\ell}\big|^{q-1}+\big|\nabla\check{\psi}_{\ell_0,\ell}\big|^{q-1}\right)\big|\nabla(\hat{\psi}_{\ell_0,\ell}
-\check{\psi}_{\ell_0,\ell})\big|.$$
Then using H\"older's inequality, we obtain
\begin{equation}
\label{seo}
I \leq  C\left( ||\nabla \hat{\psi}_{\ell_0,\ell}||_{q,D_{\ell_0}}^{\frac{q}{q'}}
+ ||\nabla \check{\psi}_{\ell_0,\ell}||_{q,D_{\ell_0}}^{\frac{q}{q'}} \right)||\nabla(\hat{\psi}_{\ell_0,\ell}-\check{\psi}_{\ell_0,\ell})||_{q,D_{\ell_0}}.
\end{equation}

Since $F$ is uniformly convex of power $q$-type, we get
\begin{eqnarray}
\label{kk}
\alpha\int_{\Omega_{\ell_0}}|\grad (u_\ell - u_\infty)|^q \leq 
\int_{\Omega_{\ell_0}} F\left(\nabla u_{\ell}\right)+F\left(\nabla u_{\infty}\right)\\ \nonumber -2F\left(\frac{\nabla u_{\ell} 
+\nabla u_{\infty}}{2}\right).
\end{eqnarray}
Combining \eqref{pp},\eqref{seo} and \eqref{kk}, one gets 
\begin{multline}\label{lm}
\alpha \int_{\Omega_{\ell_0}}|\grad (u_\ell - u_\infty)|^q \\ \leq  C\left( ||\nabla \hat{\psi}_{\ell_0,\ell}||_{q,D_{\ell_0}}^{\frac{q}{q'}}
+ ||\nabla \check{\psi}_{\ell_0,\ell}||_{q,D_{\ell_0}}^{\frac{q}{q'}} \right)||\nabla(\hat{\psi}_{\ell_0,\ell}-\check{\psi}_{\ell_0,\ell})||_{q,D_{\ell_0}}.
\end{multline}
We estimate each integral on the right hand side of the above inequality. One has for some constant $C> 0$,
\begin{multline*}
\int_{D_{\ell_0}} |\grad(\hat{\psi}_{\ell_0,\ell}-\check{\psi}_{\ell_0,\ell})|^q
= \int_{D_{\ell_0}} |\grad \{(1-\rho_{\ell_0})(u_\ell- u_\infty)\}|^q \\ \leq 
2^{q-1}\int_{D_{\ell_0}}(1- \rho_{\ell_0})^q |\grad (u_\ell - u_\infty)|^q + 2^{q-1} \int_{D_{\ell_0}} (u_\ell -u_\infty)^q
|\grad \rho_{\ell_0}|^q.
\end{multline*}

Using Poincar\'e's inequality and the properties of $\rho_{\ell_0}$,  one has for some $K=K(\lambda_1)$
\begin{equation}\label{sx}
\int_{D_{\ell_0}}|\grad(\hat{\psi}_{\ell_0,\ell}-\check{\psi}_{\ell_0,\ell})|^q\leq K \int_{D_{\ell_0}}|\grad (u_\ell-u_\infty)|^q.
\end{equation}
\smallskip

Applying Corollary \ref{za svaki skup}, one can estimate the terms $||\grad \hat{\psi}_{\ell_0,\ell}||_{q,D_{\ell_0}}$ and 
$ ||\grad \hat{\psi}_{\ell_0,\ell}||_{q,D_{\ell_0}}$ similarly, to obtain for some constant $K_1 = K_1(f, \lambda_1)$
\begin{equation*}
\int_{D_{\ell_0}}|\grad \hat{\psi}_{\ell_0,\ell}|^q, \int_{D_{\ell_0}}|\grad \check{\psi}_{\ell_0,\ell}|^q \leq K_1.
\end{equation*}
For some other constant $M$ (which depends only on $K, K_1$), \eqref{lm}   becomes
\begin{equation}
\label{kj}
\int_{\Omega_{\ell_0}}|\grad(u_\ell -u_\infty)|^q \leq M \left( \int_{D_{\ell_0}}|\grad(u_\ell -u_\infty)|^q \right)^{\frac{1}{	q}}.
\end{equation} 
Applying Corollary \ref{za svaki skup}, we get for some other constant $M_1$
\begin{equation}
\label{j}
\int_{\Omega_{\ell_0}}|\grad(u_\ell -u_\infty)|^q \leq M_1.
\end{equation}
Denoting $$a_m = \int_{\Omega_{\frac{\ell}{2}+ m}} |\grad (u_\ell -u_\infty)|^q$$
by  \eqref{kj} we have for $m = 0, \ . \ .  \ . \ ,  [\frac{\ell}{2}]-1$
\begin{equation}\label{vv}
a_m \leq M (a_{m+1} -a_m)^{\frac{1}{q}}.
\end{equation}
One may see that there exists $t_0 > 1$ such that for $1 < t < t_0$ we have 
$$\frac{1}{t^{q-1}} \leq 1 - \frac{1}{2}(q-1)(t-1).$$
It follows  that by taking $t= \frac{a_{m+1}}{a_m}$ we have that if  $\frac{a_{m+1}}{a_m} < t_0$ then
\begin{equation}\label{ddd}
(q-1)\frac{a_{m+1}-a_m}{a_m^q} \leq 2(a_m^{1-q} -a_{m+1}^{1-q}).
\end{equation}
Thus in the case $\frac{a_{m+1}}{a_m} < t_0$ by \eqref{vv} and \eqref{ddd} we have 

\begin{equation}
\label{saa}
M^{-q} \leq \frac{a_{m+1}-a_m}{a_m^q} \leq \frac{2}{q-1}(a_m^{1-q} -a_{m+1}^{1-q}).
\end{equation} 
In the case  $\frac{a_{m+1}}{a_m} > t_0$, using the bound  $a_m < M_1$ we compute 
$$a_m^{1-q} -a_{m+1}^{1-q} \geq a_m^{1-q}(1-t_0^{1-q}) = M_1(1-t_0^{1-q}).$$
Thus we have 
$$a_m^{1-q} -a_{m+1}^{1-q} \geq \min\left\{ M^{-q}\frac{q-1}{2}, M_1^{1-q}(1-t_0^{1-q}) \right\} = C_1.$$
Summing this inequality  for $m = 0, .  .  .  , [\frac{\ell}{2}]-1$
we obtain 
$$a_0^{1-q}-a_{[\frac{\ell}{2}]}^{1-q} \geq C_1[\frac{\ell}{2}].$$
Therefore it follows that 
$$\int_{\frac{\ell}{2}} |\grad (u_\ell - u_\infty)|^q = a_0 \leq \frac{1}{\left( C_1[\frac{\ell}{2}]\right)^{\frac{1}{q-1}}} \leq \frac{C_2}{\ell^{\frac{1}{q-1}}}.$$

\section{Proof of Theorem \ref{reg} and Some additional Results}

We do not restrict ourself to the assumption that $p=1$. 
 We assume that $\omega_1 $ is open and bounded subset of $\R^p$, which is star shaped around the origin. 

\begin{lemma}\label{klm}
Under the assumptions $q=2$, \eqref{alpha} and \eqref{alpha2}, one has $F\in C^1(\R^n)$.
\end{lemma}
\begin{proof}
Using the assumptions  \eqref{alpha} and \eqref{alpha2}, it is easy  to show approximating $\lap \phi$ by its discrete expression that 
$$0 \leq  \langle  \lap F ,\phi \rangle \leq 4n\beta \int_{\R^n} \phi, \hspace{3mm} \forall \phi \in C_c^\infty(\R^n), \  \phi \geq 0.$$
This implies that $\lap F$  belongs to the dual space of $L^1(\R^n)$ and hence to  $L^\infty(\R^n)$. Now the lemma follows from the estimates for Newtonian potential (see Lemma 4.1 of \cite{gil}).
\end{proof}\smallskip

Now we proceed to the proof of Theorem \ref{reg}. \smallskip

\textbf{Proof of Theorem \ref{reg}} \   We have $E_{\Omega_\ell}(u_\ell) \leq E_{\Omega_\ell}(0)= 0,$ which implies that 
$$\lambda\int_{\Omega_{\ell}} |\grad u_\ell|^2 \leq \int_{\Omega_\ell}F(\grad u_\ell) \leq  \int_{\Omega_{\ell}} fu_\ell.$$
Using H\"older's inequality and then Poincar\'e's inequality, we have for some constant $C > 0$,
\begin{equation}\label{we}\int_{\Omega_\ell} |\grad u_\ell|^2 \leq C\ell^p.
\end{equation}

For $\ell_0 < \ell - 1$, choose $\rho_{\ell_0} = \rho_{\ell_0}(X_1)$ satisfying $\rho_{\ell_0} = 1$ on $\ell_0\omega_1$ and $0$ outside $\Omega_{\ell_0+1}$.
Also assume $|\grad_{X_1}\rho_{\ell_0}| \leq C$, for some $C > 0$. 

\smallskip

Then one can proceed exactly  as in the proof of Theorem \ref{main}, until  the inequality
\begin{equation*}\label{tj}
\int_{\Omega_{\ell_0 +1}} F\left(\nabla u_{\ell}\right)+F\left(\nabla u_{\infty}\right)\leq\int_{\Omega_{\ell_0 +1}}F(\nabla\hat{\psi}_{\ell_0,\ell})+ 
F\left(\nabla \check{\psi}_{\ell_0,\ell}\right).
\end{equation*}
Using the convexity of $ q$-type  and \eqref{alpha2} we arrive to 
\begin{equation*}\alpha\int_{\Omega_{\ell_0 +1}} |\grad (u_\ell -u_\infty)|^2 \leq \beta \int_{\Omega_{\ell_0 +1}} |\grad (\hat{\psi}_{\ell_0,\ell}-\check{\psi}_{\ell_0,\ell})|^2.
\end{equation*}
Now since $\hat{\psi}_{\ell_0,\ell} - \check{\psi}_{\ell_0,\ell} = (1-\rho_{\ell_0})(u_\ell-u_\infty)$ and $\rho_{\ell_0} = 1$ on $\Omega_{\ell_0}$, this implies that 
\begin{equation*}
\alpha\int_{\Omega_{\ell_0 +1}} |\grad (u_\ell -u_\infty)|^2 \leq \beta \int_{D_{\ell_0}} |\grad (\hat{\psi}_{\ell_0,\ell}-\check{\psi}_{\ell_0,\ell})|^2.
\end{equation*}
Using  \eqref{sx} one obtains for some constant $C= C(\lambda, \lambda_1, \Lambda, \alpha, \beta)$
\begin{equation*}
\int_{\Omega_{\ell_0 }} |\grad (u_\ell -u_\infty)|^2\leq C \int_{D_{\ell_0}} |\grad (u_\ell -u_\infty)|^2,
\end{equation*}
which is equivalent to
\begin{equation}
\label{sd}
\int_{\Omega_{\ell_0 }} |\grad (u_\ell -u_\infty)|^2 \leq \frac{C}{C+1}\int_{\Omega_{\ell_0 +1 }} |\grad (u_\ell -u_\infty)|^2.
\end{equation}
Choosing $\ell_0 = \frac{\ell}{2} + m$ for $m = 0, 1, . . . , [\frac{\ell}{2}] -1$ and iterating \eqref{sd}, we get
$$\int_{\Omega_{\frac{\ell}{2}}} |\grad (u_\ell -u_\infty)|^2 \leq \left(\frac{C}{C+1}\right)^{ [\frac{\ell}{2}]-1}\int_{\Omega_{\ell}} |\grad (u_\ell -u_\infty)|^2. $$
Finally setting  $r = \frac{C}{C+1} < 1$ and from \eqref{we}, we have 
$$\int_{\Omega_{\frac{\ell}{2}}} |\grad (u_\ell -u_\infty)|^2 \leq C\ell^p e^{([\frac{\ell}{2}]-1)\log r}. $$
Since $\log r < 0$, the theorem follows.
\bigskip

  The next proposition gives a sufficient criterion for \eqref{alpha2} to hold true. More precisely we have :

\begin{proposition}\label{lip}
If $F \in C^1(\R^n)$ is such that for some $\alpha,~\beta >0$
$$
\alpha |\xi -\eta|^q \leq(\nabla F(\xi) -\nabla F(\eta))\cdot (\xi -\eta ) \hspace{4mm} \forall \xi, \eta \in \R^n
$$
or
$$
(\nabla F(\xi) -\nabla F(\eta))\cdot ( \xi -\eta ) \leq  \beta |\xi -\eta|^2, \hspace{4mm} \forall \xi, \eta \in \R^n
$$
is satisfied, then the condition \eqref{alpha} or respectively \eqref{alpha2}  holds.

\end{proposition}

\begin{proof}
One has
$$F(\xi) - F\big(\frac{\xi+ \eta}{2}\big) = -\int_0^1 \frac{d}{dt} F(\xi + t(\frac{\eta-\xi}{2}))dt$$
$$\hskip 2 cm=-\frac{1}{2}\int_0^1\nabla F(\xi + t(\frac{\eta-\xi}{2}))\cdot (\eta-\xi)dt.$$
Exchanging the roles of $\xi$ and $\eta$ we get 
$$F(\eta) - F\big(\frac{\xi+ \eta}{2}\big) = -\int_0^1 \frac{d}{dt} F(\eta + t(\frac{\xi-\eta}{2}))dt$$
$$\hskip 2 cm =-\frac{1}{2}\int_0^1\nabla F(\eta+ t(\frac{\xi-\eta}{2}))\cdot (\xi-\eta)dt.$$
Then adding the two equalities above  we obtain 
$$F(\xi) + F(\eta) - 2F\big(\frac{\xi+ \eta}{2}\big) \hskip 7 cm$$ 
$$
\hskip 2 cm= \frac{1}{2} \int_0^1(\nabla F(\eta+ t(\frac{\xi-\eta}{2})) -\nabla F(\xi + t(\frac{\eta-\xi}{2})))\cdot (\eta-\xi)dt\hskip 2 cm$$
Noting that 
$$
[\eta+ t(\frac{\xi-\eta}{2})] -[\xi + t(\frac{\eta-\xi}{2})]=(1-t)(\eta-\xi)
$$
we obtain from our assumptions
$$ \frac{\alpha}{2} \int_0^1 (1-t)^{q-1}|\eta-\xi|^qdt \leq F(\xi) + F(\eta) - 2F\big(\frac{\xi+ \eta}{2}\big)$$
or  $$F(\xi) + F(\eta) - 2F\big(\frac{\xi+ \eta}{2}\big) \leq \frac{\beta}{2} \int_0^1 (1-t)^{q-1}|\eta-\xi|^qdt$$
i.e.
$$ \frac{\alpha}{2q}|\eta-\xi|^q \leq F(\xi) + F(\eta) - 2F\big(\frac{\xi+ \eta}{2}\big)$$or

$$ F(\xi) + F(\eta) - 2F\big(\frac{\xi+ \eta}{2}\big) \leq \frac{\beta}{2q} |\eta-\xi|^q \hspace{3mm} \textrm{for}  \ q=2.$$
This completes the proof of the proposition.
\end{proof}

\bigskip

Our main result Theorem \ref{main} works only in the case when $p=1
$. Next  we  provide some partial result (Theorem \ref{p > 1 }) in the case when $0 < p < n$ and $\Omega_\ell = (-\ell , \ell)^p \times \omega_2$ are hypercubes.

\begin{lemma}\label{poin}\emph{[A pointwise estimate]}
If $f \geq 0$,  then $0 \leq u_\ell \leq u_\infty$ \ a.e. for all $\ell.$
\end{lemma}
\begin{proof}
Fix  $\ell > 0$.  First we claim that $u_\ell, u_\infty \geq 0$.  We will prove  the claim only for $u_\ell$, since the proof for  $u_\infty$ is identical. Define the function $w_\ell = \max\{0,u_\ell\}$. Clearly $w_\ell \in W_0^{1,q}(\Omega_\ell)$ is non negative and since $f\geq 0$, we have
$$E_{\Omega_\ell}(w_\ell)\leq E_{\Omega_\ell}(u_\ell).$$
The claim then follows from the uniqueness of $u_\ell$.
Set $$\mathcal{A}_\ell := \{ X \in  \Omega_\ell ~|~u_\ell(X) > u_\infty (X_2) \}.$$
We claim that  
\begin{equation*}
\label{eq444}
E_{\mathcal{A}_\ell}(u_\ell) \leq E_{\mathcal{A}_\ell}(u_\infty).
\end{equation*} 
Indeed, if not, setting $v_\ell=u_\ell -(u_\ell-u_\infty)^+$ one has $v_\ell\in W_0^{1,q}(\Omega_\ell)$ and 
$$E_{\Omega_\ell}(v_\ell)=E_{\mathcal{A}_\ell}(v_\ell) + E_{\Omega_\ell\backslash\mathcal{A}_\ell}(v_\ell) = E_{\mathcal{A}_\ell}(u_\infty) + E_{\Omega_\ell\backslash\mathcal{A}_\ell}(u_\ell)$$
$$<
E_{\mathcal{A}_\ell}(u_\ell) + E_{\Omega_\ell\backslash\mathcal{A}_\ell}(u_\ell)= E_{\Omega_\ell}(u_\ell)$$
and a contradiction with the definition of $u_\ell$. 
Setting then $w_\ell=u_\infty+(u_\ell-u_\infty)^+$ one has $w_\ell\in V^{1,q}(\Omega_\ell)$ and 
$$E_{\Omega_\ell}(w_\ell)=E_{\mathcal{A}_\ell}(w_\ell) + E_{\Omega_\ell\backslash\mathcal{A}_\ell}(w_\ell) = E_{\mathcal{A}_\ell}(u_\ell) + E_{\Omega_\ell\backslash\mathcal{A}_\ell}(u_\infty)\leq
E_{\Omega_\ell}(u_\infty).$$
Thus $w_\ell=u_\infty$ and $(u_\ell-u_\infty)^+=0$ which completes the proof.
\end{proof}
\smallskip

 Using similar argument as in the last theorem one can prove the following  monotonicity property of the 
solutions $u_\ell$. We emphasize that such a monotonicity property holds true for general domains, but we will present the result only for the  family $\Omega_\ell$.

\begin{lemma}\emph{[Monotonicity]}\label{ss}
If $f\geq 0$  and $\ell <  \ell'$ then $u_\ell \leq u_{\ell'}, \  a.e. $ in $\Omega_{\ell'}$, where $u_\ell$ is extended by $0$ outside $\Omega_{\ell'}$.
\end{lemma}

\begin{theorem}
\label{p > 1 }
Let $f \geq 0$, then for some function $\tilde{u}_\infty = \tilde{u}_\infty(X_2)$ it holds 
$$ u_\ell \rightarrow \tilde{u}_\infty(X_2)$$
pointwise.
\end{theorem}
\begin{proof}
In the statement of the theorem it is understood that $u_\ell$ are extended by $0$ outside $\Omega_\ell$.  For $h \in \R$, we set
\begin{equation*}
\label{jj}
\tau_{h}^iv := v(X-he_i) \hspace*{3mm} i = 1, .  .  .  , p
\end{equation*}
where $e_i$ denotes the unit vector in $i$-th direction.  First we claim that 
\begin{equation}\label{kl}
u_{\ell+h} \geq \tau_{h}^iu_\ell.
\end{equation}
In order to prove \eqref{kl} one has to show that  $ (\tau_h^iu_\ell - u_{\ell+h})^+ = 0.$ Set
\begin{equation*}
\label{jk}
\mathcal{A} := \left\{ X \in \Omega_{\ell+h} \ \big| \  \tau_h^iu_\ell(X) > u_{\ell+h}(X)\right\}.
\end{equation*}
 We have then $E_{\mathcal{A}}(u_{\ell+h}) \leq E_{\mathcal{A}}(\tau_h^i u_\ell)$. Indeed if this is not true, then setting 
 $$v := u_{\ell+h} + (\tau_h^iu_\ell - u_{\ell+h})^+$$
one has $v \in W_0^{1,q}(\Omega_{\ell+h})$ and
\begin{multline*}
E_{\Omega_{\ell+h}}(v) = E_{\mathcal{A}}(\tau_h^iu_\ell) + E_{\Omega_{\ell+h} \setminus \mathcal{A}}(u_{\ell+h}) \\ <  E_{\mathcal{A}}(u_{\ell+h}) + E_{\Omega_{\ell+h} \setminus \mathcal{A}}(u_{\ell+h}) = E_{\Omega_{\ell+h}}(u_{\ell+h}) \end{multline*}
and a contradiction with the definition of $u_{\ell+h}$.
\smallskip

Define $\mathcal{A'} := \left\{ X \in \Omega_{\ell} \ \big| \  \tau_{-h}^i(u_{\ell+h})(X) < u_{\ell}(X)\right\}.$ We claim that 
\begin{equation}\label{cl}
\mathcal{A} =  \mathcal{A'} +he_i.
\end{equation}
Indeed for $X$ such that $x_i < -\ell+h$ one has $\tau_h^i u_\ell = 0$ that is 
$X$ does not belongs to $\mathcal{A}$ and for $x_i-h \geq -\ell,$
$$\tau_h^i(u_\ell(X)) > u_{\ell+h}(X) \Longleftrightarrow u_\ell(X-he_i) > u_{\ell+h}(X),$$
i.e.  $X -he_i \in \mathcal{A'}$  which proves \eqref{cl}.
\smallskip

We consider then $w = u_\ell - (u_\ell -\tau_{-h}^iu_{\ell+h})^+ \in 	W_0^{1,q}(\Omega_{\ell})$. Clearly the function vanishes when $u_\ell = 0$ since $u_{\ell +h  } \geq 0.$ Then one has 

\begin{eqnarray}
E_{\Omega_\ell}(w) &=&  E_{\Omega \setminus \mathcal{A'}}(u_\ell)+ E_{\mathcal{A'}}(\tau_{-h}^iu_{\ell+h})  \nonumber\\
&=& E_{\Omega \setminus \mathcal{A'}}(u_\ell) + \int_{\mathcal{A'}}F\left(\grad u_{\ell + h}(X+he_i)\right) - fu_{\ell+h}(X + he_i) \nonumber \\
&=& E_{\Omega \setminus \mathcal{A'}}(u_\ell) + \int_{\mathcal{A}}F\left(\grad u_{\ell + h}(X)\right) - fu_{\ell+h}(X) \nonumber \\
&=& E_{\mathcal{A}}(u_{\ell +h}) + E_{\Omega \setminus \mathcal{A'}}(u_\ell) \leq
E_{\mathcal{A}}(\tau_h^iu_{\ell}) + E_{\Omega \setminus \mathcal{A'}}(u_\ell) \nonumber\\
&=& E_{\mathcal{A'}}(u_{\ell}) + E_{\Omega \setminus \mathcal{A'}}(u_\ell) = E_{\Omega_\ell}(u_\ell)\nonumber.
\end{eqnarray}
By the definition and uniqueness of $u_\ell$ this implies that 
$$w = u_\ell \Longleftrightarrow (u_\ell -\tau_{-h}^iu_{\ell+h})^+ = 0$$ and $\mathcal{A'}, \mathcal{A}$ are of measure $0$.  This proves \eqref{kl}. With similar argument one can show that 
\begin{equation}
\label{jh}
u_{\ell+ h } \geq \tau_{-h}^iu_\ell.
\end{equation}
Since $u_\ell$ is a monotone increasing sequence of functions which are bounded above by $u_\infty (X_2)$ (from Lemma \ref{poin} and Lemma \ref{ss}), one has for some $\tilde{u}_\infty$,
$$u_\ell \rightarrow \tilde{u}_\infty \hspace*{3mm}$$
pointwise. From  \eqref{jh} it follows  that 
$$\tilde{u}_\infty(X) \geq \tilde{u}_\infty(X-he_i), \ 
\tilde{u}_\infty(X) \geq \tilde{u}_\infty(X+ he_i) \hspace*{3mm} \forall h, i \in \{ 1, . . . , p\} $$
and thus $\tilde{u}_\infty $ is independent of the variable $X_1.$
\end{proof}

 We would like to point out here that it is possible to show that  $\tilde{u}_\infty = u_\infty$. We refer to M. Chipot \cite{chipot}.\smallskip
 
 Now we are interested in asymptotic behavior of the sequence $ \frac{E_{\Omega_\ell}(u_\ell)}{|\ell\omega_1|}$ as $\ell \rightarrow \infty$. ($|~~|$ denotes the measure of a set). In particular we will prove the following theorem.
\begin{theorem}\emph{[Convergence of energy]}
\label{s}
One has for some constant $C > 0$ and sufficiently large $\ell$,  $$E_{\omega_2}(u_\infty) \leq \frac{E_{\Omega_\ell}(u_\ell)}{|\ell\omega_1|} \leq  E_{\omega_2}(u_\infty) + \frac{C}{\ell}$$
where $u_\ell$ and $ u_\infty$  as in \eqref{equaation} and \eqref{equaation1} respectively.
\end{theorem}
\begin{proof}
Set
$$v_\ell(X_2) = \fint_{\ell\omega_1} u_\ell(.,X_2) = \frac{1}{|\ell\omega_1| } \int_{\ell\omega_1} u_\ell(.,X_2).$$
It is easy to see that $v_\ell \in W_0^{1,q}(\omega_2).$ Therefore
$$E_{\omega_2}(u_\infty) \leq E_{\omega_2}(v_\ell) = \int_{\omega_2} F\left(0,\grad_{X_2} v_\ell\right) -fv_\ell.$$
From divergence theorem, one has 
$$0 = \fint_{\ell\omega_1} \grad_{X_1} u_\ell$$
and by differentiation under the integral 
$$\grad_{X_2} v_\ell = \fint_{\ell\omega_1} \grad_{X_2} u_\ell.$$
Therefore we have by Jensen's inequality 
\begin{multline*}
E_{\omega_2}(u_\infty)\leq 
\int_{\omega_2} \left\{ F \left( \fint_{\ell\omega_1}\grad_{X_1} u_\ell , \ \fint_{\ell\omega_1}\grad_{X_2} u_\ell \right)  \right\}   - \int_{\omega_2} f \left\{ \fint_{\ell\omega_1}
u_\ell \right\}  \\
\leq  \int_{\omega_2}\fint_{\ell\omega_1} \left\{F(\grad u_\ell)-fu_\ell \right\} = \frac{E_{\Omega_\ell}(u_\ell)}{|\ell\omega_1|}.
\end{multline*}
This proves the first inequality. \smallskip

For the second one, first we consider a Lipschitz continuous function $\rho_\ell= \rho_\ell(X_1)$, such that $\rho_\ell = 1$ on $(\ell -1)\omega_1$ and $\rho_\ell = 0$ on $\del(\ell\omega_1)$. We also assume that there  exists a constant $C > 0 $ (independent of $\ell$) such that $$|\grad_{X_1} \rho_\ell| \leq C \ \textrm{and}   \ 0\leq \rho_\ell \leq 1.$$

\noi Thus $\rho_\ell u_\infty \in W_0^{1,q}(\Omega_\ell)$. Then from \eqref{equaation}, we have 

\begin{equation*}\label{ema}
E_{\Omega_\ell}(u_\ell) \leq E_{\Omega_\ell}(\rho_\ell u_\infty). \end{equation*}
We compute
\begin{multline*}
 E_{\Omega_\ell}(\rho_\ell u_\infty) = E_{\Omega_{\ell-1}}(u_\infty) + \int_{\Omega_\ell\setminus \Omega_{\ell-1}}F\big(\grad (\rho_\ell u_\infty) \big) -fu_\infty\rho_\ell \\ 
\leq E_{\Omega_{\ell}}(u_\infty) +  \int_{\Omega_\ell\setminus \Omega_{\ell-1}}F\big(\grad (\rho_\ell u_\infty) \big) -F\big(\grad  u_\infty \big)-fu_\infty(\rho_\ell-1)\\
 \leq  |\ell\omega_1|E_{\omega_2}(u_\infty) +  \int_{\Omega_\ell\setminus \Omega_{\ell-1}}\Lambda |\grad (\rho_\ell u_\infty) |^q + \Lambda |\grad  u_\infty |^q + |f | |u_\infty|\\
 \leq  |\ell\omega_1|E_{\omega_2}(u_\infty) +  C\int_{\Omega_\ell\setminus \Omega_{\ell-1}} |\grad  u_\infty |^q + | u_\infty |^q + |f |^{q'} \\
 \leq  |\ell\omega_1|E_{\omega_2}(u_\infty) +   C \{ |\ell\omega_1| - |(\ell-1)\omega_1|\} \int_{\omega_2} |\grad  u_\infty |^q +| u_\infty |^q + |f |^{q'}.
 \end{multline*}
Dividing  by $|\ell\omega_1|$ and the result follows, i.e. one has 
$$\frac{E_{\Omega_\ell}(u_\ell)}{|\ell\omega_1|} \leq  E_{\omega_2}(u_\infty) + \frac{C}{\ell}.$$

This finishes the proof of the theorem.
\end{proof}
\bigskip

\textbf{Acknowledgment}  The authors of this paper would like to thank the anonymous refree for his/her valuable suggetions in improving the quality of this work. The argument from \eqref{vv} till the end of the  proof of our main theorem is due to him which provided a better rate of convergence. The research leading to these results has received funding from 
 Lithuanian-Swiss cooperation programme to reduce economic and social 
 disparities within the enlarged European Union under project agreement 
 No CH-3-SMM-01/0.  The research of the first author was supported also by  the Swiss National
Science Foundation under the contracts $\#$ 200021-129807/1 and 200021-146620.
The research of the second author was supported 
 by the European Initial Training Network  FIRST under the grant
  agreement \# PITN-GA-2009-238702. The research of the third author is funded by ``Innovation in Science Pursuit for Inspired Research(INSPIRE)"
under the IVR Number: 20140000099.

\end{document}